\documentclass[11pt,reqno,a4paper]{amsart}
\usepackage[margin=0.75in]{geometry}
\usepackage[usenames]{color}
\usepackage[usenames]{color}
\usepackage{amsmath,pdfsync,verbatim,graphicx,epstopdf,enumerate}
\usepackage[colorlinks=true]{hyperref}
\usepackage{cancel}
\usepackage[framemethod=tikz]{mdframed}
\hypersetup{allcolors=blue}
\allowdisplaybreaks
\usepackage{bm}

\newcommand{\wh}{\widehat}

\newcommand{\lb}{\left(}

\newcommand{\rb}{\right)}
\newcommand{\PD}{\partial}
\renewcommand{\d}{\delta}
\newcommand{\Beq}{\begin{equation}}
	\newcommand{\Eeq}{\end{equation}}
\newcommand{\beq}{\begin{equation*}}
	\newcommand{\eeq}{\end{equation*}}
\newcommand{\bal}{\begin{align}}
	\newcommand{\eal}{\end{align}}

\usepackage{mathtools}

\usepackage[notref,notcite]{}

\newcommand{\bp}{\begin{prob}}
	\newcommand{\ep}{\end{prob}}
\newcommand{\bpr}{\begin{proof}}
	\newcommand{\epr}{\end{proof}}

\newcommand{\bel}[1]{\begin{equation}\label{#1}}
	\newcommand{\ee}{\end{equation}}

\newtheorem{theorem}{Theorem}[section]
\newtheorem{lemma}[theorem]{Lemma}

\newtheorem{question}[theorem]{Question}

\newtheorem{corollary}[theorem]{Corollary}

\theoremstyle{definition}

\theoremstyle{remark}
\newtheorem{remark}[theorem]{Remark}

\numberwithin{equation}{section}

\newcommand{\Rn}{\mathbb{R}^n}
\newcommand{\R}{\mathbb{R}}
\newcommand{\D}{\mathrm{d}}
\newcommand{\Lc}{\mathcal{L}}

\newcommand{\Dc}{\mathcal{D}}
\newcommand{\Nc}{\mathcal{N}}
\newcommand{\A}{\alpha}

\def\tbl{\textcolor{blue}}

\usepackage[normalem]{ulem}

\newcommand{\Bb}{\mathbb{B}}

\usepackage[usenames]{color}
\usepackage{amsmath,pdfsync,verbatim,graphicx,epstopdf,enumerate}
\renewcommand{\d}{\delta}

\newcommand{\Rc}{\mathcal{R}}
\newcommand{\Sc}{\mathcal{S}}

\newcommand{\Nb}{\mathbb{N}}

\newcommand{\Sb}{\mathbb{S}}
\newcommand{\Sn}{\mathbb{S}^{n-1}}

\title{$d$-plane transform: unique and non-unique continuation}
\author{Divyansh Agrawal and Nisha Singhal}
\date{\today}

\address {Centre for Applicable Mathematics, Tata Institute of Fundamental Research, Bangalore, India
\newline
E-mail:{\tt\ agrdiv01@gmail.com, agrawald@tifrbng.res.in; nisha2020@tifrbng.res.in}
\newline
Orcid:{\tt\ 0009-0003-5125-0640, 0009-0006-3005-1986}
}
\begin{document}
\begin{abstract}
The $d$-plane transform maps functions to their integrals over $d$-planes in $\Rn$. We study the following question: if a function vanishes in a bounded open set, and its $d$-plane transform vanishes on all $d$-planes intersecting the same set, does the function vanish identically? For $d$ an even integer, we show  by producing an explicit counterexample, that neither the $d$-plane transform, nor its normal operator has this property. On the other hand, an even stronger property holds when $d$ is odd, where the normal operator vanishing to infinite order at a point, along with the function vanishing on an open set containing that point, is sufficient to conclude that the function vanishes identically.
\end{abstract}
\maketitle
\section{Introduction}
The \emph{$d$-plane transform} (also called the \emph{$k$-plane transform} or the \emph{generalized Radon transform}) maps a function to its integrals over given $d$-planes in $\Rn$, $0<d<n$. For the integral to make sense, some decay of the function at infinity is required. We restrict our attention to the space of compactly supported smooth functions in $\Rn$. For $d=1$, the integration takes place over straight lines in $\Rn$, and the transform is then called the X-ray transform, or simply the ray transform, due to its applications in radiology. For $d=n-1$, the integration takes place over hyperplanes in $\Rn$ and it is then the well-known Radon transform. In this article, we study the so-called unique continuation results for the $d$-plane transform, for all $d \in \{1, \dots, n-1\}$. The unique continuation property (henceforth abbreviated to UCP) for the $d$-plane transform asks the following question:

\begin{question}
    If a function and its $d$-plane transform both vanish on a non-empty bounded open subset of $\Rn$, does the function vanish identically?
\end{question}

\noindent Here the $d$-plane transform vanishing on an open set is to be interpreted as vanishing on all $d$-planes which intersect the open set. We study UCP for the $d$-plane transform, as well as its normal operator, and give complete answers in both the cases. We show that if $d$ is even, UCP does not hold for either the transform or its normal operator. We in fact produce an explicit counterexample. We also prove that when $d$ is odd, a stronger form of UCP holds (see Theorem~\ref{ucp-nd-inf}), which improves upon the previously known results (for instance, those of \cite{Covi_Railo_Mönkkönen}). The dichotomy between the behavior of $d$-plane transform when $d$ is even and odd can be attributed to the presence of powers of the Laplacian in its inversion formula. A similar phenomenon was also observed in the context of local equatorial characterization of zonoids (\cite{Nazarov})\footnote{We thank the reviewer for bringing this to our attention.}, where the authors showed that there is no local equatorial characterization in odd dimensions, while such was known to be true in even dimensions (\cite{Goodey-Weil1,Panina}). Interestingly, that was a consequence of the fact that the inversion formula for the cosine transform is non-local in odd dimensions.

The $d$-plane transform has a long history and has been studied by many authors in various settings \cite{Berenstein-inversion,Berenstein-range,Estrada-Rubin,Rubin-inv,Rubin-inj}. It is known to be injective when the transform is  known over all $d$-planes in $\Rn$ and explicit inversion formulas exist \cite{Helgason_Book}. The injectivity has also been studied in restricted settings \cite{Rub-restr}. A complete description of the range of the transform can also be found in \cite{Helgason_Book}. A study of the $d$-plane transform of functions with low regularity can be found in \cite{Rubin-wavelet-inv,Rubin-recon} and references therein. We refer the reader to \cite{Dr1,Rub-norm,Rubin-est} and the references therein for a study of the mapping properties of the transform between weighted $L^p$-spaces.

UCP in the context of integral geometry has gained considerable interest in the past few years. These results are possible due to their close connection with powers of the Laplacian, $(-\Delta)^r$, which is a non-local operator when $r$ is not an integer. In \cite{Keijo_partial_function, Keijo_partial_vector_field}, the authors study UCP for normal operator of the ray transform, acting on functions and vector fields, respectively. These results were generalized to normal operator of ray and moment ray transforms of tensor fields in \cite{AKS}. In \cite{ilmavirta2023unique}, the authors study the measurable UCP (the function $f$ vanishes in some open set $U$ and the normal operator is only assumed to vanish in a set $E \subset U$ of positive measure) for moment ray transform and study the UCP for fractional moment transforms. 

In \cite{Covi_Railo_Mönkkönen}, the authors study UCP for higher fractional powers of the Laplacian, and obtain as a corollary, the (weak)-UCP (see Subsection \ref{defnucp} for the definition of weak and strong UCP) for the normal operator of the $d$-plane transform, when $d$ is odd. In the same work, the authors raise the question as to whether UCP can hold when $d$ is even. The current article addresses both these situations. As mentioned before, the $d$-plane transform coincides with ray and Radon transforms for $d=1$ and $d=n-1$, respectively. Thus, all the previous results on UCP for ray and Radon transforms of functions follow from the results of the current article. 

Besides proving the stronger form of UCP than what was previously known for odd integers $d$, the main novelty of this work is the construction of explicit counterexamples to  UCP when $d$ is even. Let us mention some important counterexamples in the integral geometry literature for Radon and related transforms. After local injectivity of the weighted Radon transform with positive real-analytic weights was proved in \cite{Boman-Quinto-Duke}, it was believed that a similar result should also be true for positive smooth weights, since the former class is dense in the latter. However, Boman \cite{Boman-counterexample} gave a counterexample and showed that even global injectivity could fail when the weight is only a positive smooth function. In fact, the set of weights for which local injectivity fails was also shown to be dense in the space of positive smooth weights \cite{Boman-ce}. In \cite{Zalcman-ce}, the author discussed uniqueness of the Radon transform in the plane under various restrictions on the space of the lines and gave various examples of non-uniqueness. In \cite{agrawal2023simple}, the authors showed that the UCP does not hold for the spherical Radon transform in odd dimensions.

The rest of the article is organized as follows. In Section \ref{main_results}, we state our main results. These are divided into two subsections. Subsection \ref{lucp} contains our results which show that UCP can not hold when $d$ is even, and Subsection \ref{ucp} contains the results about UCP when $d$ is odd. We introduce our notation and give relevant definitions briefly in Section \ref{prelim}. The proofs of the results are collected in Section \ref{proof}. 

\section{Main Results} \label{main_results}
In this article, we always assume $n \geq 2$. For $0 < d < n$, a $d$-dimensional affine subspace of $\Rn$ is called a \emph{$d$-plane in $\Rn$}. Let $G(d,n)$ denote the set of all $d$-planes in $\Rn$. We denote by $G_{d,n}$ the Grassmann manifold of all $d$-dimensional subspaces ($d$-planes through the origin) of $\Rn$. The parallel translation of a $d$-plane to one through the origin gives a mapping $\pi$ of $G(d,n)$ onto $G_{d,n}$. The inverse image $\pi^{-1}(\sigma)$ of a member $\sigma \in G_{d,n}$ is then identified with the orthogonal complement $\sigma^\perp$. Thus, any $d$-plane in $\Rn$ can be uniquely parametrized as
\[
\xi = (\sigma, x'') = \sigma + x'', \quad \sigma = \pi(\xi)~~ \text{and} ~x'' = \sigma^\perp \cap \xi.
\]

Let $C_c^\infty(\Rn)$ denote the space of compactly supported smooth functions in $\Rn$. Let $\Rc_d$ denote the $d$-plane transform which maps a function to its integral over a given $d$-plane. In the case when $d=n-1$, this is the well-known Radon transform, which we will denote by $\Rc$. One can consider the formal adjoint of $\Rc_d$ and introduce a backprojection operator $\Rc_d^*$. The composition of these operators leads to the normal operator of the $d$-plane transform, denoted henceforth by $\Nc_d$, and simply by $\Nc$ in the case of Radon transform ($d=n-1$).
See Section \ref{prelim} for precise definitions and expressions for the operators and Subsection \ref{defnucp} for our definitions of weak and strong UCP in the context of the $d$-plane transform. We now state our main results in the next two subsections.

\subsection{Lack of UCP when the $d$-plane is even dimensional}\label{lucp}
Our first set of results shows that neither weak nor strong unique continuation can hold for the $d$-plane transform, when $d$ is an even natural number. 

\begin{theorem}\label{lucp-rd}
    Let $0 < d < n$ be an even integer. Let $U$ be a bounded open subset of $\Rn$. There exists a non-trivial function $f\in C_c^{\infty}(\Rn)$ such that $f|_U = 0$ and $\Rc_d f(\xi) = 0$ for all $d$-planes $\xi$ such that $\xi \cap U \neq \phi$.
\end{theorem}
\noindent Our proof of the theorem is constructive, i.e., we explicitly produce the function that serves as a counterexample. This also answers a question raised in \cite{Covi_Railo_Mönkkönen} (see discussion after corollary 2 there), as to whether UCP holds when $d$ is even. As a corollary, we obtain that UCP does not hold for the normal operator of the $d$-plane transform. 
\begin{corollary}\label{lucp-nd}
    Let $0 < d < n$ be even and $U \subset \Rn$ be open and bounded. There exists a non-trivial function $f \in C_c^\infty(\Rn)$ such that $f\lvert_U = 0$ and  $\Nc_d f|_U = 0$. 
\end{corollary}
\noindent This can also be proven directly using the fact that the normal operator is a convolution with the Riesz potential (see Subsection~\ref{p-lucp}).

Recall that a function $g \in C^\infty(\Rn)$ is said to \emph{vanish to infinite order at a point $x_0 \in \Rn$} if $\lb \PD^\A g \rb(x_0) = 0$ for all multi-indices $\A$. Note that if a smooth function vanishes in an open set, then it vanishes to infinite order at each point of the open set. Thus we also have the existence of a non-trivial function $f \in C_c^\infty(\Rn)$ such that $f$ vanishes in an open set and $\Nc_d f$ vanishes to infinite order at a point in the set.
\noindent Let us mention the case of Radon transform separately. The reason for doing so is that we will first construct a function as a counterexample for the Radon transform, and then use that to construct the function for the $d$-plane transform.
\begin{theorem}\label{lucp-rad}
    Let $n \geq 2$ be an odd integer. Let $U$ be a bounded open subset of $\Rn$. There exists a non-trivial function $f \in C_c^\infty(\Rn)$ such that $f|_U = 0$ and $\Rc f (\xi) = 0$ for all hyperplanes $\xi$ intersecting $U$.
\end{theorem}

\begin{remark}
    Note that all the theorems stated in this subsection require $U$ to be bounded. The following is an interesting example suggested by the reviewer for when it is not bounded. Let $U$ be an open set such that any $d$-plane intersects $U$ (for instance, $U$ can be taken to be an open neighborhood of the coordinate hyperplanes). Then $\Rc_d f \equiv 0$ trivially yields $f \equiv 0$.
\end{remark}

\subsection{UCP when the $d$-plane is odd dimensional}\label{ucp}
Our next set of results shows that when $d$ is odd, the strong UCP holds for the normal operator $\Nc_d$, which in turn implies the weak UCP and the UCP  for $\Rc_d$.

\begin{theorem}\label{ucp-nd-inf}
    Let $0 < d < n$ be an odd integer. Let $U$ be an open subset of $\Rn$ and $x_0 \in U$. Let $f \in C_c^\infty(\Rn)$ be such that $f|_U = 0$ and $\Nc_d f$ vanishes to infinite order at $x_0$. Then $f \equiv 0$. 
\end{theorem}

\begin{remark}
    The above result can also be proved for the normal operator $\Nc_d$ acting on compactly supported distributions following the details in \cite{Keijo_partial_function}, and invoking Corollary~\ref{poly} which removes the restriction on $d$ in \cite[Lemma~2.2]{Keijo_partial_function}.
\end{remark}

If $\Nc_d f \lvert_U = 0$, then trivially $\Nc_d f$ vanishes to infinite order at each point of $U$. Thus, the strong UCP for the normal operator (Theorem~\ref{ucp-nd-inf}) implies the weak UCP: if $f$ and $\Nc_d f$ both vanish in an open set $U$, then $f$ vanishes identically. The weak UCP for the normal operator was obtained in \cite{Covi_Railo_Mönkkönen} as a consequence of UCP for fractional Laplacian.

Similarly, if $\Rc_d f$ vanishes on all $d$-planes passing through an open set $U$, then $\Nc_d f$ vanishes in $U$ (see \eqref{adjoint} and \eqref{normal}). Thus, we obtain the UCP for the $d$-plane transform: if $f$ and $\Rc_d f$ vanish on an open set $U$, then $f$ vanishes identically.

\noindent The case of Radon transform is of particular interest, which we state separately. In this case, we also present a different proof, using the expansion of the function into a sum of spherical harmonics.
\begin{theorem}\label{ucp-rad}
    Let $n \geq 2$ be an even integer and $f \in C_c^{\infty}(\Rn)$. Let $U$ be an open subset of $\Rn$ such that $f|_U = 0$ and $\Rc f (\xi) = 0$ for all hyperplanes $\xi$ intersecting $U$. Then $f \equiv 0$.
\end{theorem}

\section{Preliminaries}\label{prelim}
This section is devoted to fixing our notation and giving relevant definitions. Let $\Rn$ denote the $n$-dimensional Euclidean space with the open unit ball denoted as $\Bb$ and let $\Sn$ denote the unit sphere. We let $C^\infty(\Rn)$ denote the space of smooth functions and $C_c^\infty(\Rn)$ denote its subspace of functions having compact support. We let $\Sc(\Rn)$ denote the space of Schwartz functions.

\subsection{The \textit{d-}plane transform} In this subsection, we give a brief introduction to the $d$-plane transform, mostly following the classical book by Helgason \cite[\tbl{p. 32-34}]{Helgason_Book}. Let $d$ be an integer such that $0 < d < n$, and let $G(d,n)$ denote the manifold of all $d$-planes in $\Rn$. The $d$-plane transform of a function $f$ is defined as
\begin{equation}
    \Rc_d f (\xi) = \int\limits_{\xi} f(x) \D m(x),
\end{equation}
where $\xi \in G(d,n)$ is a $d$-plane in $\Rn$, and $\D m$ denotes the natural $d$-dimensional measure on $\xi$. When $d=n-1$, it is called the Radon transform and simply denoted as $\Rc$; and in the case of $d=1$, it coincides with the so-called X-ray transform. The $d$-plane transform of a continuous function is well-defined under some mild decay conditions at infinity. We will restrict our attention to functions in $C_c^\infty(\Rn)$, in which case, $\Rc_d: C_c^\infty(\Rn) \to C_c^\infty(G(d,n))$. 

It is advantageous to study the formal dual transform of $\Rc_d$, given as
\begin{equation}\label{adjoint}
    \Rc_d^*: C^\infty(G(d,n)) \to C^\infty(\Rn), \quad
    \Rc_d^* \varphi (x) = \int\limits_{x \in \xi} \varphi(\xi) ~\D \mu(\xi),
\end{equation}
where $\mu$ is the unique measure on the compact space of $d$-planes passing through $x$, invariant under all rotations around $x$ and with total measure $1$. Another utility of studying the adjoint operator is to extend the definition of the transform to the space of distributions. We do not consider this extension here, and refer the interested reader to the article \cite{Covi_Railo_Mönkkönen}. The composition of these operators gives rise to the so-called normal operator, $\Nc_d \coloneqq \Rc_d^* \Rc_d : C_c^\infty(\Rn) \to C^\infty(\Rn)$, which has the following expression:
\begin{equation}\label{normal}
    \begin{split}
        \Nc_d f(x) &= \frac{|\Sb^{d-1}|}{|\Sb^{n-1}|} \int\limits_{\Rn} f(y) \frac{1}{|x-y|^{n-d}} ~\D y \\
        &= \frac{|\Sb^{d-1}|}{|\Sb^{n-1}|} \lb f \ast \frac{1}{|\cdot|^{n-d}} \rb (x), 
    \end{split}
\end{equation}
where $|\Sb^{d-1}|$ and $|\Sn|$ denote the surface areas of the unit spheres in $d$- and $n$-dimensions, respectively. Note that this expression can be used to define the normal operator for compactly supported distributions as well. 

\subsection{Unique continuation for the \textit{d}-plane transform}\label{defnucp} Let us now explain what we mean by \emph{unique continuation} of the $d$-plane transform. Let $U \subset \Rn$ be an open set and let $\Lc$ denote either the $d$-plane transform or its normal operator. We say that the operator $\Lc$ has unique continuation property (UCP) if there does not exist a non-zero function $f$ which vanishes in $U$ and $\Lc f$ vanishes in $U$. Thus, if $\Lc$ has the unique continuation property and if $f$ and $\Lc f$ both vanish in an open set, then $f$ must vanish identically. If $\Lc$ is the $d$-plane transform $\Rc_d$, then $\Rc_d f$ vanishing on an open set $U$ means that $\Rc_d f$ vanishes on all $d$-planes intersecting $U$. We sometimes refer to this property as the weak UCP, since we also study the strong UCP: $\Lc$ is said to have strong UCP if for any $f$ such that $f$ vanishes in an open set $U$ and $\Lc f$ vanishes to infinite order at a point in $U$, then $f$ vanishes identically.

\subsection{Spherical harmonics and Gegenbauer polynomials}\label{shgp} Spherical harmonics are special functions which arise in many theoretical and practical applications. We only provide a brief introduction, fix our notation and state some results used in the article. We refer the interested reader to \cite{Stein-Weiss} for a comprehensive treatment (see also \cite{Müller_book, Seeley}). Spherical harmonics are restrictions to the unit sphere of homogeneous harmonic polynomials. For a fixed $l \geq 0$, there are 
\[
N(n,l) = \frac{(2l+n-2)(n+l-3)!}{l! (n-2)!}, \quad N(n,0) = 1,
\]
linearly independent spherical harmonics of degree $l$ in $\Rn$, which we denote by $\{Y_{l,k}\}$. These form a complete set of orthonormal functions on the sphere (i.e., the closure of the span of the spherical harmonics is $L^2(\Sb^{n-1})$), and any function $f \in C_c^\infty(\Rn)$ can be written as 
\[
f(x)=\sum\limits_{l=0}^{\infty} \sum\limits_{k=0}^{N(n,l)} f_{l, k}(|x|)Y_{l, k}\lb \frac{x}{|x|} \rb,
\]
where 
\[
f_{l, k}(|x|)=f_{l,k}(r)=\int\limits_{\Sb^{n-1}} f(r\theta) \overline{Y}_{l,k}(\theta) \D \theta.
\]

A closely related family of special functions is \emph{Gegenbauer polynomials}, denoted as $C_l^\A$. For $\A > -1/2$, $C_l^\A$ is a polynomial of degree $l$, and the family of Gegenbauer polynomials is orthogonal on $[-1,1]$ with the weight $(1-x^2)^{\A-1/2}$. We assume that these are normalized so that $C_l^\A(1) = 1$.  We will be mainly concerned with the case $\A= \frac{n-2}{2}$. For $n=2$, these reduce to the Chebyshev polynomials. We will need the Rodrigues' formula for Gegenbauer polynomials:
\begin{equation}\label{geg}
    C_l^\A (x) = K \lb 1-x^2 \rb^{-\A + 1/2} \frac{\D^l}{\D x^l} \lb 1-x^2 \rb^{l+\A-1/2},
\end{equation}
where $K = K(l,\A)$ is a constant.

\section{Proofs of main results}\label{proof}
\subsection{Lack of UCP when the $d$-plane is even dimensional}\label{p-lucp} In this subsection,  we present the proofs of the results in the Subsection \ref{lucp}, concerning the lack of UCP. Before we proceed to the proofs, we make the following trivial observation. Using translation and scaling, we can assume $U$ to be the unit ball in $\Rn$, and $x_0$ can be taken to be the origin. This can be done by considering a large enough ball containing the set $U$, followed by translation and scaling, if required. In all the proofs below, we assume this simplification. 

We start by presenting a proof of Theorem~\ref{lucp-rad}.
\begin{proof}[Proof of Theorem~\ref{lucp-rad}]
    Let $n \geq 3$ be odd. Consider a non-zero function $h \in C^\infty(\R)$ with the following properties:
\begin{itemize}
	\item $h$ is even,
	\item $\mathrm{supp} (h) \subset (-a, -1) \cup (1,a)$ for some $ a > 1$.
\end{itemize}
Define $g(\theta, s) = h(s)$. By range conditions and support theorem for Radon transform \cite[p. 37]{Natterer_book}, there exists $f \in C_c^\infty(\Rn)$ such that $f$ is supported in the ball of radius $a$ centered at the origin, and $\Rc f = g$. Since $g$ is rotation invariant, $f$ is radial and the inversion formula for radial functions \cite[p. 26]{Natterer_book} gives for $|x| < 1$ and $s > 0$:
\begin{align*}
f(x) &= c(n) |x|^{2-n} \int\limits_{|x|}^\infty (s^2 - |x|^2)^{(n-3)/2} g^{(n-1)}(s) \mathrm{d}s \\
&= c(n) |x|^{2-n} \int\limits_{1}^{a} (s^2 - |x|^2)^{(n-3)/2} h^{(n-1)}(s) \mathrm{d}s.
\end{align*} 
Since $n$ is odd, $\lb s^2 - |x|^2 \rb^{(n-3)/2}$ is a polynomial in $s$, of degree at most $n-3$. Using integration by parts and by the choice of $h$, it is immediate that $f(x) = 0$ for $|x| < 1$, but $f$ does not vanish identically. 
\end{proof}

\noindent We have shown that there exists a compactly supported smooth radial function in $\Rn$ which vanishes in the unit ball, and whose Radon transform vanishes on all hyperplanes intersecting the unit ball. This is a crucial ingredient in the proof for Theorem~\ref{lucp-rd}.
\begin{proof}[Proof of Theorem~\ref{lucp-rd}]
    Let $0 <d < n$ be even. If $d=n-1$, this becomes Theorem~\ref{lucp-rad}, and has been proved above. Let us assume $ 0 < d \leq n-2$. By Theorem~\ref{lucp-rad}, there exists a non-trivial radial function $f_d: \R^{d+1} \to \R$ satisfying the following: 
    \begin{itemize}
        \item $f_d \in C_c^\infty(\R^{d+1})$,
        \item $f_d$ does not vanish identically,
        \item $f_d$ vanishes in the unit ball of $\R^{d+1}$,
        \item $\Rc_d f_d (\xi_d) = 0$, for any $d$-plane $\xi_d$ in $\R^{d+1}$ which intersects the unit ball of $\R^{d+1}$.
    \end{itemize}
    Let us define the function $f: \Rn \to \R$, by extending the above function radially, i.e.,
    \[
    f(x_1, \dots x_n) \coloneqq f_d ( |x|, \underbrace{ 0, \dots, 0}_{\text{$d$-zeros}} ).
    \]
    From the corresponding properties of $f_d$, it is easy to see that the function $f$ has the following properties:
    \begin{itemize}
        \item $f \in C_c^\infty(\Rn)$,
        \item $f$ does not vanish identically,
        \item $f$ vanishes in the unit ball of $\Rn$.
    \end{itemize}
    
    All that is left to show is that $\Rc_d f (\xi) = 0$, for all $d$-planes intersecting the unit ball in $\Rn$.
    Let $\xi$ be such a $d$-plane in $\Rn$. Then there exist orthonormal vectors $v_0, \dots, v_d$ in $\Rn$ such that $\xi$ can be written as
    \[
    \xi = \left\{a v_0 + \sum\limits_{i=1}^d c_i v_i~ \Big\lvert~ c_i \in \R \right\}
    \]
    It then follows that $a = \mathrm{dist}(0,\xi) < 1$. Since $v_0, \dots, v_d$ are orthonormal, they are linearly independent as well. Extend $v_0, \dots, v_d$ to an orthonormal basis for $\Rn$, say $\{v_0, \dots, v_d, u_1, \dots , u_{n-d-1}\}$. The matrix $A$ defined as
    \[
    A = \begin{bmatrix}
        v_0 & v_1 & \dots & v_d & u_1 & \dots & u_{n-d-1}
    \end{bmatrix}^T
    \]
    belongs to  $O(n)$, the group of orthogonal matrices, and satisfies  
    \[
    A v_j =  e_j, \quad \text{for all} ~~ 0 \leq j \leq d.
    \]
    Then the image of $\xi$ under $A$ lies in $\R^{d+1}$, canonically embedded in $\Rn$. We have
    \begin{align*}
        \Rc_d f (\xi) &= \int\limits_\xi f(x)~ \D x \\
        &= \int\limits_{A\xi} f\lb A^T y \rb \D y \\
        &= \int\limits_{A\xi} f_d(|y|, 0, \dots, 0) ~\D y \\
        &= 0,
    \end{align*}
    since $A\xi$ is a $d$-plane in $\R^{d+1}$ intersecting the unit ball of $\R^{d+1}$. This finishes the proof.
\end{proof}
Notice that if $\Rc_d f = 0$ for all $d$-planes which intersect the unit ball, then the normal operator also vanishes in the unit ball. Hence, Corollary~\ref{lucp-nd} follows from Theorem~\ref{lucp-rd}. However, we could not resist presenting the proof below, due to its remarkable simplicity, and the crucial use of the locality of integer powers of Laplacian. 

\begin{proof}[Proof of Corollary~\ref{lucp-nd}]
Consider a non-trivial function $u \in C_c^\infty(\Rn)$ such that $u\lvert_\Bb = 0$ and define $f = (-\Delta)^{d/2} u$. Then $f$ vanishes in $\Bb$ and $f \in C_c^\infty(\Rn)$. Since $u$ is non-trivial and has compact support, $u$ is not polyharmonic and hence $f$ does not vanish identically. By taking Fourier transform (viewing $\frac{1}{|x|^{n-d}} \in L^1_{\mathrm{loc}}(\Rn)$ as a tempered distribution) in the equation \eqref{normal} and using the expression for Fourier transform of Riesz potential \cite[Eq 53, p. 241]{Helgason_Book}, we have (up to dimensional constants)
\begin{align*}
    \widehat{\Nc_d f}(\xi) &= |\xi|^{-d} \widehat{f}(\xi) \\
    &= \widehat{u}(\xi),
\end{align*}
which implies that $\Nc_d f  = u$ (up to constants). Hence $f$ is the required function.
\end{proof}
\subsection{UCP when the $d$-plane is odd dimensional}\label{p-ucp}
We now provide the proofs of the results in Subsection \ref{ucp}. These results show that UCP holds when the surface of integration is odd dimensional. 
We assume that $0 < d < n$ and that $d$ is odd. As before, we also assume that $U = \Bb$, the unit ball centered at the origin in $\Rn$, and $x_0 = 0$.

We start with a density lemma, which forms a crucial ingredient in the proof. 
\begin{lemma}
    Let $0 < d < n$ be odd. For each $m \geq 0$ and $1 \leq i_1, \dots, i_m \leq n$, the term
    \[
        \frac{x_{i_1} x_{i_2} \dots x_{i_m}}{|x|^{2m+2k+n-d}}
    \]
    can be written as a finite linear combination of derivatives of $\frac{1}{|x|^{n-d}}$, for all $k \geq 0$.
\end{lemma}

\begin{proof}
    We prove the statement by induction on $m$. We have 
    \[
        \Delta \lb \frac{1}{|x|^{n-d}}\rb = -(n-d)(d-2) \frac{1}{|x|^{n-d+2}}, 
    \]
    where each of the coefficients is non-zero for $d$ odd. Taking higher powers of $\Delta$, the result is proved for $m=0$. 
    
    Now let us assume the statement is true for some $m$, and consider
    \begin{align*}
        \PD_{i_{m+1}} \lb \frac{x_{i_1} \dots x_{i_m}}{|x|^{2m+2k+n-d}} \rb &= -(2m+2k+n-d) \frac{x_{i_1} \dots x_{i_{m+1}}}{|x|^{2(m+1)+2k+n-d}} \\
        & \qquad \qquad \qquad + \sum\limits_{j=1}^m \d_{i_{m+1} \, i_j} \, \frac{x_{i_1} \dots \wh{x_{i_j}} \dots x_{i_m}}{|x|^{2m+2k+n-d}},
    \end{align*}
where the $\,\wh{}\,$ denotes the missing term. The term on the left and each term in the sum on the right can be written as a finite linear combination of derivatives of $\frac{1}{|x|^{n-d}}$, by induction hypothesis. Since $d < n$, $2m+2k+n-d$ is never zero, and the claim follows.
\end{proof}

We will only need the following corollary, obtained by taking $k=0$ above.
\begin{corollary}\label{poly}
    Let $n \geq 2$ and $0 < d < n$ be an odd integer. For any polynomial $p$, $p\lb \frac{x}{|x|^2}\rb \cdot \frac{1}{|x|^{n-d}}$ can be written as a finite linear combination of derivatives of $\frac{1}{|x|^{n-d}}$. 
\end{corollary}

The above lemma refines \cite[Lemma~2.2]{Keijo_partial_function}. The restriction on the order of Riesz potentials in the results of \cite{Keijo_partial_function} arose from the limitation in the lemma, which is removed here.

\begin{proof}[Proof of Theorem~\ref{ucp-nd-inf}]
    Let us denote the constant in the expression \eqref{normal} by $C$. We have
    \begin{align*}
        \Nc_d f (x) &= C \lb f \ast \frac{1}{|x|^{n-d}} \rb (x).
    \end{align*}
    Let $p$ be a polynomial, and $\Dc$ denote the differential operator such that 
    \[
    \Dc \lb \frac{1}{|x|^{n-d}} \rb = p \lb \frac{x}{|x|^2}\rb \cdot \frac{1}{|x|^{n-d}}.
    \]
    Such a differential operator (of finite order) exists by Corollary~\ref{poly}. Therefore, we get
    \begin{align*}
        \Dc \Nc_d f (x) &= C~ \Dc \lb f \ast \frac{1}{|x|^{n-d}} \rb (x) \\
        &= C \lb f \ast  p \lb \frac{x}{|x|^2}\rb \cdot \frac{1}{|x|^{n-d}} \rb (x).
    \end{align*}
    The above expression is well-defined for $|x|$ small enough, since $f$ vanishes in the unit ball. Let $R > 1$ be such that $\mathrm{supp}(f) \subset B(0,R)$. Since $\Nc_d f$ vanishes to infinite order at the origin, we get
    \begin{align*}
        0 &= \int\limits_{B(0,R)} f(-y)  p \lb \frac{y}{|y|^2}\rb \cdot \frac{1}{|y|^{n-d}} \,\D y \\
        &= \int\limits_{A(0,1,R)} f(-y)  p \lb \frac{y}{|y|^2}\rb \cdot \frac{1}{|y|^{n-d}} \,\D y ,
    \end{align*}
    since $f$ vanishes in the unit ball. Here, $A(0,r_1,r_2)$ denotes the annulus centered at origin, with inner and outer radii $r_1$ and $r_2$, respectively. Considering the Kelvin transform $x = \frac{y}{|y|^2}$, we have
    \begin{align*}
        0 &= \int\limits_{A(0, 1/R,1)} f \lb -\frac{x}{|x|^2} \rb p(x) |x|^{n-d} \frac{1}{|x|^{2n}} \,\D x \\
        &= \int\limits_{A(0,1/R,1)} \Tilde{f}(x) p(x) \,\D x,
    \end{align*}
    where $\Tilde{f}(x) = f \lb -\frac{x}{|x|^2} \rb  |x|^{-n-d}$, and hence $\mathrm{supp}(\Tilde{f}) \subset A(0,1/R,1)$, and $\Tilde{f}$ is smooth. Since the polynomial $p$ is arbitrary, we conclude that $\Tilde{f}$, and hence $f$ vanishes identically.
\end{proof}

Before we proceed to the proof of Theorem~\ref{ucp-rad}, we need the following density result.

\begin{lemma}\label{density}
    Let $0 < a < b < \infty$ and  $f \in C_c^\infty((a,b))$ be such that for some $k \in \Nb \cup \{0\}$,
    \[
    \int\limits_a^b f(r) \lb 1- \frac{s^2}{r^2} \rb^{\frac{2k-1}{2}} \D r =0
    \]
    for all $s \in [0,a)$. Then $f \equiv 0$.
\end{lemma}

\begin{proof}
    Since $f \equiv 0$ if and only if $\frac{f(r)}{r^{2k}}$ vanishes identically, it is enough to consider $k=0$. This can be seen by repeated application of the operator $D = \frac{1}{s} \frac{\D}{\D s}$. In this case, the identity can be written as
    \begin{align*}
        0 &= \int\limits_a^b \frac{r f(r)}{\sqrt{r^2 - s^2}} \, \D r.
    \end{align*}
    By repeated application of the operator $D$, we find that for any polynomial $p$,
    \begin{align*}
        0 &= \int\limits_a^b \frac{r f(r)}{\sqrt{r^2 - s^2}} p \lb \frac{1}{r^2-s^2} \rb\, \D r.
    \end{align*}
    Evaluating the expression at $s=0$, we get
    \begin{align*}
        0 &= \int\limits_a^b f(r) p \lb \frac{1}{r^2} \rb \D r.
    \end{align*}
    By the transformation $ u = \frac{1}{r^2}$, we arrive at 
    \begin{align*}
        \int\limits_{b^{-2}}^{a^{-2}} \Tilde{f}(u) p(u) \D u & = 0,
    \end{align*}
    for any polynomial $p$, where $\Tilde{f}(u) \coloneqq u^{-3/2} f(u^{-1/2})$. We conclude that $\Tilde{f}$ and hence $f$ vanishes identically.
\end{proof}

\begin{proof}[Proof of Theorem~\ref{ucp-rad}]
    As before, we assume $U = \Bb$ and $f\in C_c^{\infty}(B(1,R))$, for some $R > 1$. Let us denote $\Rc f$ by $g$. Then $f$ and $g$ can both be expanded into spherical harmonics:
    \beq
        f(r,\theta)=\sum\limits_{l=0}^{\infty} \sum\limits_{k=0}^{N(n,l)} f_{l, k}(r)Y_{l, k}(\theta),
    \eeq
    \beq
        g(s,\theta)=\sum\limits_{l=0}^{\infty} \sum\limits_{k=0}^{N(n,l)} g_{l, k}(s)Y_{l, k}(\theta),
    \eeq
where $N(n,l)$ denotes the number of linearly independent spherical harmonics of degree $l$ in $\Rn$. It is known that the coefficients in the spherical harmonics expansion of $f$ and $g$ satisfy \cite{Natterer_book}
\[
g_{l,k}(s) = |\Sb^{n-2}| \int\limits_{s}^\infty r^{n-2} f_{l,k}(r) C_l^\A \lb \frac{s}{r} \rb \lb 1 - \frac{s^2}{r^2} \rb^{(n-3)/2} \D r,
\]
where $C_l^\A$ denote the Gegenbauer polynomials, defined in Subsection~\ref{shgp}, with $\A = (n-2)/2$.

Since $f$ vanishes in $\Bb$ and outside $B(0,R)$, we must have $f_{l,k}(r) = 0$ for $r \leq 1$ and $r\geq R$. Similarly, since $g$ vanishes for $s < 1$, we find that for $s < 1$
\begin{align*}
    0&=g_{l, k}(s)\\
    &=|\Sb^{n-2}|\int\limits_{1}^{R} r^{n-2} f_{l,k}(r) C_l^\A \lb \frac{s}{r} \rb \lb 1 - \frac{s^2}{r^2} \rb^{(n-3)/2} \D r.
\end{align*}
Recall that the Gegenbauer polynomials have the expression \eqref{geg}:
\begin{align*}
    C_l^\A (x) &= K (1-x^2) \frac{\D^l}{\D x^l} (1-x^2)^{l+\A-\frac{1}{2}}.
\end{align*}
Using chain rule, we obtain
\begin{align*}
    C_l^\A \lb \frac{s}{r} \rb &= K \lb 1-\frac{s^2}{r^2} \rb^{-\A+\frac{1}{2}} \lb \frac{-r^2}{s} \frac{\D}{\D r} \rb^l \lb 1-\frac{s^2}{r^2} \rb^{l+\A-\frac{1}{2}}.
\end{align*}
Thus, we obtain for all $s < 1$:
\begin{align*}
    0 &= \int\limits_{1}^R r^{n-2} f_{l,k}(r) \lb \frac{-r^2}{s} \frac{\D}{\D r} \rb^l \lb 1-\frac{s^2}{r^2} \rb^{l+\A-\frac{1}{2}} \, \D r.
\end{align*}
Notice that the terms $\frac{1}{s^l}$ can be taken out of the integral. Integrating by parts $l-$times, we find
\begin{align*}
    0 &= \int\limits_1^R \frac{1}{r^2} \lb r^2 \frac{\D}{\D r} \rb^l \lb r^n f_{l,k}(r) \rb \lb 1-\frac{s^2}{r^2} \rb^{l+\A-\frac{1}{2}} \, \D r.
\end{align*}
Invoking Lemma~\ref{density}, we obtain $\frac{1}{r^2} \lb r^2 \frac{\D}{\D r} \rb^l \lb r^n f_{l,k}(r) \rb$ vanishes identically, and hence so does $f_{l,k}$, since it is compactly supported. Hence, $f$ vanishes identically.
\end{proof}
\section{Acknowledgments}
The authors thank Venky Krishnan and Gaik Ambartsoumian for various helpful discussions. The authors are also grateful to the anonymous referees for helpful comments which led to the improvement in the presentation of the article.

\bibliographystyle{amsplain}
\bibliography{reference}

\providecommand{\bysame}{\leavevmode\hbox to3em{\hrulefill}\thinspace}
\providecommand{\MR}{\relax\ifhmode\unskip\space\fi MR }
\providecommand{\MRhref}[2]{%
  \href{http://www.ams.org/mathscinet-getitem?mr=#1}{#2}
}
\providecommand{\href}[2]{#2}
\begin{thebibliography}{10}

\bibitem{agrawal2023simple}
Divyansh Agrawal, Gaik Ambartsoumian, Venkateswaran~P. Krishnan, and Nisha Singhal, \emph{A simple range characterization for spherical mean transform in odd dimensions and its applications}, 2023, Preprint: arxiv.org/abs/2310.20702.

\bibitem{AKS}
Divyansh Agrawal, Venkateswaran~P. Krishnan, and Suman~Kumar Sahoo, \emph{Unique continuation results for certain generalized ray transforms of symmetric tensor fields}, J. Geom. Anal. \textbf{32} (2022), no.~10, Paper No. 245, 27. \MR{4456212}

\bibitem{Berenstein-inversion}
Carlos~A. Berenstein and Enrico Casadio~Tarabusi, \emph{Inversion formulas for the {$k$}-dimensional {R}adon transform in real hyperbolic spaces}, Duke Math. J. \textbf{62} (1991), no.~3, 613--631. \MR{1104811}

\bibitem{Berenstein-range}
Carlos~A. Berenstein and Enrico~Casadio Tarabusi, \emph{Range of the {$k$}-dimensional {R}adon transform in real hyperbolic spaces}, Forum Math. \textbf{5} (1993), no.~6, 603--616. \MR{1242891}

\bibitem{Boman-counterexample}
Jan Boman, \emph{An example of nonuniqueness for a generalized {R}adon transform}, J. Anal. Math. \textbf{61} (1993), 395--401. \MR{1253450}

\bibitem{Boman-ce}
\bysame, \emph{Local non-injectivity for weighted {R}adon transforms}, Tomography and inverse transport theory, Contemp. Math., vol. 559, Amer. Math. Soc., Providence, RI, 2011, pp.~39--47. \MR{2885194}

\bibitem{Boman-Quinto-Duke}
Jan Boman and Eric~Todd Quinto, \emph{Support theorems for real-analytic {R}adon transforms}, Duke Math. J. \textbf{55} (1987), no.~4, 943--948. \MR{916130}

\bibitem{Covi_Railo_Mönkkönen}
Giovanni Covi, Keijo M\"{o}nkk\"{o}nen, and Jesse Railo, \emph{Unique continuation property and {P}oincar\'{e} inequality for higher order fractional {L}aplacians with applications in inverse problems}, Inverse Probl. Imaging \textbf{15} (2021), no.~4, 641--681. \MR{4259671}

\bibitem{Dr1}
Alexis Drouot, \emph{Quantitative form of certain {$k$}-plane transform inequalities}, J. Funct. Anal. \textbf{268} (2015), no.~5, 1241--1276. \MR{3304599}

\bibitem{Estrada-Rubin}
Ricardo Estrada and Boris Rubin, \emph{Radon-{J}ohn transforms and spherical harmonics}, Representation theory and harmonic analysis on symmetric spaces, Contemp. Math., vol. 714, Amer. Math. Soc., [Providence], RI, [2018] \copyright 2018, pp.~131--142. \MR{3847247}

\bibitem{Goodey-Weil1}
Paul Goodey and Wolfgang Weil, \emph{Centrally symmetric convex bodies and the spherical {R}adon transform}, J. Differential Geom. \textbf{35} (1992), no.~3, 675--688. \MR{1163454}

\bibitem{Helgason_Book}
Sigurdur Helgason, \emph{The {R}adon transform}, second ed., Progress in Mathematics, vol.~5, Birkh\"auser Boston, Inc., Boston, MA, 1999. \MR{1723736}

\bibitem{ilmavirta2023unique}
Joonas Ilmavirta, Pu-Zhao Kow, and Suman~Kumar Sahoo, \emph{Unique continuation for the momentum ray transform}, 2023.

\bibitem{Keijo_partial_function}
Joonas Ilmavirta and Keijo M\"{o}nkk\"{o}nen, \emph{Unique continuation of the normal operator of the x-ray transform and applications in geophysics}, Inverse Problems \textbf{36} (2020), no.~4, 045014, 23. \MR{4103726}

\bibitem{Keijo_partial_vector_field}
\bysame, \emph{X-ray tomography of one-forms with partial data}, SIAM J. Math. Anal. \textbf{53} (2021), no.~3, 3002--3015. \MR{4261111}

\bibitem{Müller_book}
Claus M\"{u}ller, \emph{Spherical harmonics}, Lecture Notes in Mathematics, vol.~17, Springer-Verlag, Berlin-New York, 1966. \MR{199449}

\bibitem{Natterer_book}
F.~Natterer, \emph{The mathematics of computerized tomography}, Classics in Applied Mathematics, vol.~32, Society for Industrial and Applied Mathematics (SIAM), Philadelphia, PA, 2001, Reprint of the 1986 original. \MR{1847845}

\bibitem{Nazarov}
Fedor Nazarov, Dmitry Ryabogin, and Artem Zvavitch, \emph{On the local equatorial characterization of zonoids and intersection bodies}, Adv. Math. \textbf{217} (2008), no.~3, 1368--1380. \MR{2383902}

\bibitem{Panina}
G.~Yu. Panina, \emph{Representation of an {$n$}-dimensional body in the form of a sum of {$(n-1)$}-dimensional bodies}, Izv. Akad. Nauk Armyan. SSR Ser. Mat. \textbf{23} (1988), no.~4, 385--395, 409. \MR{997401}

\bibitem{Rubin-wavelet-inv}
B.~Rubin, \emph{Inversion of {$k$}-plane transforms via continuous wavelet transforms}, J. Math. Anal. Appl. \textbf{220} (1998), no.~1, 187--203. \MR{1613940}

\bibitem{Rub-norm}
\bysame, \emph{Weighted norm inequalities for {$k$}-plane transforms}, Proc. Amer. Math. Soc. \textbf{142} (2014), no.~10, 3455--3467. \MR{3238421}

\bibitem{Rub-restr}
\bysame, \emph{Overdetermined transforms in integral geometry}, Complex analysis and dynamical systems {VI}. {P}art 1, Contemp. Math., vol. 653, Amer. Math. Soc., Providence, RI, 2015, pp.~291--313. \MR{3453082}

\bibitem{Rubin-est}
\bysame, \emph{Norm estimates for {$k$}-plane transforms and geometric inequalities}, Adv. Math. \textbf{349} (2019), 29--55. \MR{3937741}

\bibitem{Rubin-recon}
Boris Rubin, \emph{Reconstruction of functions from their integrals over {$k$}-planes}, Israel J. Math. \textbf{141} (2004), 93--117. \MR{2063027}

\bibitem{Rubin-inv}
\bysame, \emph{On some inversion formulas for {R}iesz potentials and {$k$}-plane transforms}, Fract. Calc. Appl. Anal. \textbf{15} (2012), no.~1, 34--43. \MR{2872110}

\bibitem{Rubin-inj}
\bysame, \emph{On the injectivity of integral operators related to the {E}uler-{P}oisson-{D}arboux equation and shifted {$k$}-plane transforms}, Anal. Math. Phys. \textbf{13} (2023), no.~4, Paper No. 56, 15. \MR{4601149}

\bibitem{Seeley}
R.~T. Seeley, \emph{Spherical harmonics}, Amer. Math. Monthly \textbf{73} (1966), no.~4, 115--121. \MR{201695}

\bibitem{Stein-Weiss}
Elias~M. Stein and Guido Weiss, \emph{Introduction to {F}ourier analysis on {E}uclidean spaces}, Princeton Mathematical Series, vol. No. 32, Princeton University Press, Princeton, NJ, 1971. \MR{304972}

\bibitem{Zalcman-ce}
Lawrence Zalcman, \emph{Uniqueness and nonuniqueness for the {R}adon transform}, Bull. London Math. Soc. \textbf{14} (1982), no.~3, 241--245. \MR{656606}

\end{thebibliography}

\end{document}